\documentclass[12pt]{amsart}
\usepackage{amsmath,amssymb,amscd,latexsym}
\usepackage[all]{xy}

\usepackage[usenames, dvipsnames]{color}
\usepackage{enumitem}

\newcommand{\bi}{\begin{itemize}}
\newcommand{\ei}{\end{itemize}}
\newcommand{\bn}{\begin{enumerate}}
\newcommand{\en}{\end{enumerate}}

\newcommand{\bq}{\begin{equation}}
\newcommand{\eq}{\end{equation}}

\newcommand{\A}{{\mathbb{A}}}
\newcommand{\C}{{\mathbb{C}}}

\newcommand{\MGLet}{MGL_{\et}}

\newcommand{\MUlog}{MU_{\log}}

\newcommand{\Ge}{\mathbb{G}}

\newcommand{\Pro}{{\mathbb{P}}}

\newcommand{\Z}{{\mathbb{Z}}}
\newcommand{\et}{\mathrm{\acute{e}t}}

\newcommand{\Nis}{\mathrm{Nis}}

\newcommand{\Spec}{\mathrm{Spec}\,}

\newcommand{\colim}{\operatorname*{colim}}
\newcommand{\hocolim}{\operatorname*{hocolim}}

\newcommand{\Hom}{\mathrm{Hom}}

\newcommand{\Map}{\mathrm{Map}}

\newcommand{\Ub}{\mathrm{U}_{\bullet}}
\newcommand{\Xb}{\mathrm{X}_{\bullet}}
\newcommand{\Sm}{\mathbf{Sm}}
\newcommand{\Smc}{\Sm_{\C}}

\newcommand{\SmS}{\Sm_{S}}

\newcommand{\Spc}{\mathbf{Spc}}
\newcommand{\Spcp}{\mathbf{Spc}_{\ast}}
\newcommand{\Spt}{\mathbf{Spt}}

\newcommand{\sS}{\mathbf{sS}}

\newcommand{\Eh}{{\mathcal E}}

\newcommand{\Hh}{{\mathcal H}}
\newcommand{\Hhp}{{\mathcal H}_{\ast}}

\newcommand{\Hhstau}{\Hh_{s,\tau}}
\newcommand{\HhsNis}{\Hh_{s,\Nis}}
\newcommand{\Hhset}{\Hh_{s,\et}}

\newcommand{\Hhpstau}{\Hh_{\bullet s,\tau}}

\newcommand{\SHh}{{\mathcal{SH}}}

\newcommand{\Xh}{\mathcal{X}}

\newcommand{\Yh}{\mathcal{Y}}

\newtheorem{theorem}{Theorem}[section]
\newtheorem{lemma}[theorem]{Lemma}
\newtheorem{prop}[theorem]{Proposition}

\theoremstyle{definition}

\newtheorem{remark}[theorem]{Remark}

\begin{document}

\title{Verdier Hypercovering Theorem for motivic spectra}

\author{Gereon Quick}
\thanks{Both authors were supported in part by the German Research Foundation (DFG) under RO 4754/1-1}
\address{Department of Mathematical Sciences, NTNU Trondheim, Norway}
\email{gereon.quick@math.ntnu.no}

\author{Andreas Rosenschon}
\address{Mathematisches Institut, LMU M\"unchen, Germany}
\email{axr@math.lmu.de}

\date{}

%\maketitle

\begin{abstract}
We prove a Verdier Hypercovering Theorem for cohomology theories arising from motivic spectra. This allows us to construct for smooth quasi-projective complex varieties a natural morphism from 
\'etale algebraic to Hodge filtered complex cobordism, which extends the map from \'etale motivic
to Deligne-Beilinson cohomology.
\end{abstract}

\maketitle

\section{Introduction}
Let $X$ be a smooth quasi-projective complex variety, and let $H^m_M(X;\Z(n))$ be the motivic 
cohomology groups, defined as the hypercohomology groups of Bloch's cycle complex, viewed as 
a complex of Zariski sheaves (or equivalently, as the hypercohomology groups of Voevodsky's 
complex $\Z(n)$). Since these complexes are also complexes of \'etale sheaves, we have the 
analogously defined 
\'etale motivic cohomology groups $H^m_L(X;\Z(n))$, together with an evident map $H^m_M(X;\Z(n))
\rightarrow H^m_L(X;\Z(n))$. It is known that with rational coefficients this comparison map is 
an isomorphism; however, with integral coefficients these groups are different in general. For 
example, there is a map $c^{n}_{L,B}: CH^n_L(X)=H^{2n}_L(X;\Z(n))\rightarrow H^{2n}_B(X;\Z(n))$ 
from the \'etale Chow groups to singular cohomology, which is surjective on torsion 
\cite[Theorem 1.1]{rs}. Because of the counterexamples to the integral Hodge conjecture given 
by Atiyah-Hirzebruch \cite{AH}, this implies that $CH^n_L(X)$ contains more elements than the 
usual Chow group $CH^n(X)$, and that $c^n_{L,B}$ cannot arise in the usual fashion as a cycle 
map coming from a cycle on $X$. To give a geometric interpretation of the \'etale motivic 
cohomology groups and to define more general maps from \'etale motivic cohomology to other 
cohomology theories, it has been shown in 
\cite[Theorem 4.2]{rs} that the elements of $H^m_L(X,\Z(n))$ have an interpretation in terms of 
cycles on \'etale covers of $X$; more precisely, there is an isomorphism 
\bq\label{Hcolimintro}
\colim_{\Ub \to X} H^m_M(\Ub; \Z(n)) \xrightarrow{\cong} H^m_L(X; \Z(n))
\eq
where the colimit runs over all \'etale hypercovers of $X$. The proof of this result in 
\cite[\S 4]{rs} uses rather sophisticated techniques and relies on the proof of the 
Beilinson-Lichtenbaum conjecture by Voevodsky \cite{Vo-2} and Rost-Voevodsky \cite{Vo-p}. 

In this note, we first use homotopy-theoretic methods to prove the above type of Verdier 
Hypercovering Theorem in a far more general context for cohomology theories arising 
from motivic spectra:

\begin{theorem}\label{verdierE}
Let $X$ be a smooth quasi-projective scheme over a Noetherian scheme $S$. If 
$E$ is a motivic spectrum over $S$ and $\Ub\rightarrow X$ is an \'etale hypercover, let 
$E^{m,n}(\Ub)$ and $E_{\et}^{m,n}(X)$ be the motivic and the \'etale motivic 
$E$-cohomology groups of $\Ub$ and $X$ respectively. Then there is a natural isomorphism 
\[
\colim_{\Ub \to X} E^{m,n}(\Ub) \overset{\cong}{\to} E_{\et}^{m,n}(X),
\]
where the colimit runs over all \'etale hypercovers of $X$. 
\end{theorem}

Taking $E=H\Z$ and $S=\Spec(k)$ for a field $k$, it follows that 
the isomorphism (\ref{Hcolimintro}) holds for a smooth quasi-projective variety 
over a field, independent of further assumption such as, for example, finite 
cohomological dimension. 

The isomorphism (\ref{Hcolimintro}) has been used in \cite{rs} to construct a 
map from \'etale motivic to Deligne-Beilinson cohomology $c^{m,n}_{L,D}: 
H^m_L(X;Z(n))\rightarrow H^m_D(X;\Z(n))$, where Deligne-Beilinson cohomology is 
defined as the hypercohomology of a complex of Zariski sheaves \cite{EV}. If $X$ is 
projective, there is an isomorphism
\begin{equation}\label{db-deligne}
H^m_D(X;\Z(n))\cong H^m(X;\Z_D(n)),
\end{equation}
where the group on the right is the cohomology of the analytic Deligne complex $\Z_D(n)$, 
which is quasi-isomorphic to the homotopy pullback of the diagram of complexes of sheaves
arising from the inclusions $\Omega^{\geq n}_X \rightarrow \Omega^{\bullet}_X\leftarrow \Z$.  
In \cite{hfc} variants of Deligne cohomology theories have been constructed by replacing
the complex $\Z$ (which represents singular cohomology) with a spectrum representing a 
more general cohomology theory. In particular, this construction applied to the Thom
spectrum $MU$ yields the Hodge filtered cobordism groups $MU_{log}^m(n)(X)$ with the 
property that the map $MU\rightarrow H\Z$ induces natural homomorphisms 
$MU^m_{log}(n)(X)\rightarrow H^m(X;\Z_D(n))$. Since filtered Hodge cobordism is an oriented
motivic cohomology theory, the universal property of algebraic cobordism represented by 
the motivic spectrum $MGL$ yields maps
\begin{equation}\label{algcob-fil}
MGL^{m,n}(X) \rightarrow \MUlog^{m}(n)(X). 
\end{equation}

We use Theorem \ref{verdierE} to show the following. 

\begin{theorem}\label{thm-maps}
Let $X$ be a smooth quasi-projective complex variety and let $m,n$ be integers. 
Then there are natural homomorphisms 
\begin{equation}\label{thm-maps-1claim}
\MGLet^{m,n}(X)  \to \MUlog^{m}(n)(X)
\end{equation}
such that
$MGL^{m,n}(X) \rightarrow \MGLet^{m,n}(X) \rightarrow \MUlog^{m}(n)(X)$ 
coincides with (\ref{algcob-fil}). 
If $X$ is projective, the map $MGL \to H\Z$ induces a natural 
commutative diagram
\bq\label{natdiagram}
\xymatrix{
\MGLet^{m,n}(X) \ar[d] \ar[r] & \MUlog^{m}(n)(X) \ar[d] \\
H_L^m(X; \Z(n)) \ar[r] & H^m(X; \Z_D(n)).}
\eq
\end{theorem}
We remark that for a smooth projective complex variety the restriction of the 
map in the bottom row (\ref{natdiagram}) to torsion subgroups in an isomorphism, 
provided $m\neq 2n$ \cite[Theorem 1.2]{rs}. It is tempting to ask whether the 
restriction of the top row to torsion is an isomorphism as well, allowing to determine
the torsion in \'etale cobordism groups via filtered Hodge cobordism.

\section{Verdier's hypercovering theorem for motivic spectra}

\subsection{Preliminaries}
Let $\SmS$ be the category of smooth schemes over a Noetherian scheme $S$, and let 
$\Spc(S)$ be the category of simplicial presheaves on $\SmS$. Thus objects of
$\Spc(S)$ are contravariant functors from $\SmS$ to the category $\sS$ of simplicial 
sets, which we refer to as spaces (over $S$). 
%There are several important model structures 
%on the category $\Spc(S)$. In this paper we will work in the local injective model 
%structure (see \cite{jardine}, for example). 
Let $f \colon \Xh \to \Yh$ be a morphism of spaces. Then $f$ is called
\begin{itemize}[leftmargin=*]
\item
a projective weak equivalence, if it induces a weak equivalence of simplicial 
sets $\Xh(U) \to \Yh(U)$ for every object $U$ of $\SmS$;
\item 
a projective fibration, if it induces a Kan fibration of simplicial sets 
$\Xh(U) \to \Yh(U)$ for every object $U$ of $\Sm/S$;
\item
 a projective cofibration, if it has the right lifting property with respect 
to any acyclic projective fibration.
\end{itemize}
These classes of morphisms define a closed model structure on $\Spc(S)$, called the 
projective model structure (see \cite{dugger}).

We will consider $\SmS$ as a site with respect to a Grothendieck topology $\tau$. 
To obtain a model structure which is sensitive to the topology $\tau$, one needs 
to modify the above structure. We will consider the cases when $\tau=\Nis$ is the 
Nisnevich topology or $\tau=\et$ is the \'etale topology. Then $f \colon \Xh \to \Yh$ 
is called
\begin{itemize}[leftmargin=*]
\item a $\tau$-weak equivalence (or just weak equivalence), if it induces a weak 
equivalence of simplicial sets $\Xh_x \to \Yh_x$ at every $\tau$-point $x$ of the 
site $\SmS$;
\item a $\tau$-cofibration (or cofibration), if it is a projective cofibration;
\item a $\tau$-local projective fibration (or local projective fibration), if it 
has the right lifting property with respect to any projective cofibration which is 
also a weak equivalence.
\end{itemize}
These classes of morphisms define a closed proper cellular simplicial model structure on $\Spc(S)$, 
the local projective model structure (see \cite{dugger} and \cite[Theorem 2.3]{jardine} for the 
corresponding injective structure which is Quillen equivalent to the projective one). 
Since we will only use this projective structure, we will often omit the word `projective'. 
Let $\Hhstau(S)$ be the homotopy category of $\Spc(S)$, considered as a site with respect to 
$\tau$. The category $\Spcp(S)$ of pointed spaces over $S$ has a model structure via the 
forgetful functor $\Spcp(S) \to \Spc(S)$ and we write $\Hhpstau(S)$ for the corresponding
homotopy category. 

Dugger, Hollander and Isaksen \cite{dhi} have shown that one way to obtain the local projective model 
structure is to form the localization of the projective model structure with respect to the special 
class of morphisms called hypercovers. Since these hypercovers will play an essential role in this paper, 
we will recall their definition following the conventions used in \cite{dhi}: Given a topology $\tau$ 
on $\SmS$, a map $f$ of simplicial presheaves is called a stalkwise fibration (resp.~acyclic stalkwise 
fibration), if the map of stalks $f_x$ is a Kan fibration (resp.~Kan fibration and weak equivalence) of 
simplicial sets for every $\tau$-point $x$. Let $X$ be an object of $\SmS$ and let $\Ub$ be a simplicial 
presheaf, together with an augmentation map $\Ub \to X$ in $\Spc(S)$. This map is called a 
$\tau$-hypercover of $X$ if it is an acyclic stalkwise fibration and each $U_n$ is a coproduct 
of representables. Note that the projective model structure on $\Spc(S)$ has the property that every 
hypercover is a morphism of cofibrant objects \cite{dugger}.  
Moreover, by \cite{dhi} the fibrations in the local projective model structure on $\Spc(S)$ 
admit a characterization in terms of such hypercovers. 
%Let $\Ub \to X$ be a hypercover. 
Following \cite{dhi}, we say that 
a simplicial presheaf $\Yh$ satisfies descent for a hypercover $\Ub \to X$, if there is a projective 
fibrant replacement $\Yh \to \Yh'$ with the property that the natural map 
\bq\label{descentmap}
\Map(X, \Yh') \to \Map(\Ub, \Yh')
\eq
is a weak equivalence of simplicial sets, where $\Map$ denotes the mapping space in the simplicial 
structure on spaces. It is easy to see that if $\Yh$ satisfies descent for a hypercover $\Ub \to X$, 
then the map \eqref{descentmap} is a weak equivalence for {\it every} objectwise fibrant replacement 
$\Yh'$. Moreover, the local projective fibrant objects in $\Spc(S)$ are exactly those spaces which 
are projective fibrant and satisfy descent with respect to all hypercovers \cite[Corollary 7.1]{dhi}. 

\subsection{The classical case.} Let $\tau$ be either the Nisnevich or the \'etale topology on $\SmS$. 
For stalkwise fibrant spaces $\Xh$ and $\Yh$, simplicial homotopy of maps $\Xh \to \Yh$ is an equivalence 
relation. The set  $\pi(\Xh,\Yh)$ of simplicial homotopy classes of morphisms from $\Xh$ to $\Yh$ is 
the quotient of $\Hom_{\Spc(S)}(\Xh,\Yh)$ with respect to the equivalence relation generated by 
simplicial homotopies. For $X\in \SmS$,  we write $\pi HC_{\tau}/X$ for the category whose objects are 
the $\tau$-hypercovers of $X$ and whose morphisms are simplicial homotopy classes of morphisms which 
fit in the obvious commutative triangle over $X$. The category $\pi HC_{\tau}/X$ is filtered 
(see \cite[Proposition 8.5]{dhi}, for instance). A crucial observation, made first by Brown 
\cite[Proof of Theorem 2]{brown}, is that one can use $\pi HC_{\tau}/X$ to approximate the homotopy 
category $\Hhstau(S)$ in the following sense, yielding a generalization of the Verdier 
Hypercovering Theorem \cite[expos\'e V, 7.4.1(4)]{AGV} (see also \cite[Theorem 8.6]{dhi} and 
\cite{jardineverdier}):  

\begin{theorem}\label{verdier}
Let $\Yh$ be a stalkwise fibrant simplicial presheaf and let $X$ be an object in $\SmS$. Then the 
canonical map induces a bijection
\[
\colim_{\Ub \to X \in \pi HC_{\tau}/X} \pi(\Ub, \Yh) \overset{\cong}{\to} \Hom_{\Hhstau(S)}(X, \Yh). 
\]
\end{theorem}

We apply Theorem \ref{verdier} to obtain a description of
$\Hom_{\Hhset(S)}(X, \Yh)$,~i.e. the set of maps between a smooth scheme $X$ over $S$ 
and a projective fibrant space $\Yh$ in the \'etale homotopy category $\Hhset(S)$. 
Since $\Yh$ is also stalkwise fibrant for the Nisnevich and the \'etale topology
on $\Sm/S$, we have from Theorem \ref{verdier}
\bq\label{verdierapplied}
\colim_{\Ub \to X \in \pi HC_{\et}/X} \pi(\Ub, \Yh) \overset{\cong}{\to} \Hom_{\Hhset(S)}(X, \Yh).
\eq

Let $\Yh$ be a fibrant object in the local Nisnevich model structure. Then $\Ub$ is a 
cofibrant object, and the set $\pi(\Ub, \Yh)$ of simplicial homotopy classes of maps is 
in bijection with the set of morphisms from $\Ub$ to $\Yh$ in the homotopy category 
associated with local Nisnevich model structure on $\Spc(S)_{\Nis}$. In particular, we 
obtain from \eqref{verdierapplied} the following bijection
\bq\label{verdierappliedNis}
\colim_{\Ub \to X \in \pi HC_{\et}/X} \Hom_{\HhsNis(S)}(\Ub, \Yh) \overset{\cong}{\to} 
\Hom_{\Hhset(S)}(X, \Yh).\eq

\subsection{A motivic variant.} We prove a motivic analogue of (\ref{verdierappliedNis}). 
Let $\Yh$ be a simplicial presheaf on $\SmS$. Recall that $\Yh$ is $\A^1$-local, if for
every object for every $X \in \SmS$ the projection $X\times_S \A_S^1 \to X$ induces a weak equivalence 
\bq\label{localmap}
\Map(X, \Yh) \to \Map(X\times_S \A_S^1, \Yh). 
\eq
If $\Yh$ is $\A^1$-local, then $\Yh$ is Nisnevich $\A^1$-local (resp.~\'etale $\A^1$-local), if 
$\Yh$ is Nisnevich local fibrant (resp.~\'etale local fibrant).

Since the motivic model structure is given by a left Bousfield localization with respect to the maps
$X\times_S \A^1_S \to X$ for all $X\in \SmS$, it follows that the Nisnevich $\A^1$-local objects 
(resp.~\'etale $\A^1$-local objects) are exactly the fibrant objects in the Nisnevich motivic 
structure (resp.~\'etale motivic model structure) in $\Spc(S)$. 
Let $\Hh_{Nis}(S)$ (resp.~$\Hh_{\et}(S)$) be the motivic homotopy category of spaces with respect to 
the Nisnevich topology (resp. \'etale topology). 

\begin{lemma}\label{lemmaA1local}
Let $\Yh$ be a simplicial presheaf on $\SmS$ which is Nisnevich $\A^1$-local. Then a fibrant replacement 
of $\Yh$ in the \'etale local model structure is an \'etale-$\A^1$-local simplicial presheaf. In particular, 
for $X\in \SmS$ we have
\bq\label{HomNistoet}
\Hom_{\Hhset(S)}(X, \Yh) \xrightarrow{\cong} \Hom_{\Hh_{\et}(S)}(X, \Yh).  
\eq
\end{lemma}

\begin{proof}
Let $X \in \SmS$ and let $q \colon \Yh \to R_{\et}\Yh$ be an acyclic cofibration in the 
\'etale local model structure with the property that $R_{\et}\Yh$ is \'etale local fibrant. 
Then $q$ induces the following commutative diagram
\[
\xymatrix{
\Map(X, \Yh) \ar[d] \ar[r] & \Map(X\times_S \A_S^1, \Yh) \ar[d] \\
\Map(X, R_{et}\Yh) \ar[r] & \Map(X\times_S \A_S^1, R_{et}\Yh).}
\]
By assumption $\Yh$ is Nisnevich $\A^1$-local, hence the top horizontal map is 
a weak equivalence. Since $q$ is an acyclic cofibration and all objects are cofibrant, 
we also know that the two vertical maps are weak equivalences. Hence the lower horizontal 
map is a weak equivalence as well, and $R_{et}\Yh$ is \'etale-$\A^1$-local. 
For the second assertion note that since $R_{\et}\Yh$ is \'etale $\A^1$-local, 
the diagonal maps in the commutative diagram
\[
\xymatrix{
 & \ar[dl]_{\cong} \pi(X, R_{\et}\Yh) \ar[dr]^{\cong} & \\
\Hom_{\Hhset(S)}(X, \Yh) \ar[rr] & & \Hom_{\Hh_{\et}(S)}(X, \Yh).}
\]
are bijections. Thus the bottom row is a bijection, which proves \eqref{HomNistoet}.
\end{proof}

The next Proposition gives the motivic analogue of (\ref{verdierappliedNis}): 

\begin{prop}\label{verdierA1prop}
Let $X \in \SmS$ and let $\Yh$ be a simplicial presheaf which is Nisnevich-$\A^1$-local. Then the natural map
induces a bijection
\bq\label{verdierappliedNisA1}
\colim_{\Ub \to X \in \pi HC_{\et}/X} \Hom_{\Hh_{\Nis}(S)}(\Ub, \Yh) \overset{\cong}{\to} 
\Hom_{\Hh_{\et}(S)}(X, \Yh). 
\eq
\end{prop}

\begin{proof}
Because $\Yh$ is an objectwise fibrant simplicial presheaf, by Theorem \ref{verdier}
\[
\colim_{\Ub \to X \in \pi HC_{\et}/X} \pi(\Ub, \Yh) \overset{\cong}{\to} \Hom_{\Hhset(S)}(X, \Yh). 
\]
Since $\Yh$ is $\A^1$-local and Nisnevich local fibrant, it is a fibrant object in the Nisnevich 
motivic local model structure on $\Spc(S)$. Since all objects in this structure are cofibrant, 
the set of simplicial homotopy classes $\pi(\Ub, \Yh)$ computes the set of morphisms in the 
motivic Nisnevich homotopy category,~i.e.
\[
\pi(\Ub, \Yh) \cong \Hom_{\Hh_{\Nis}(S)}(\Ub, \Yh).
\]
Again, since $\Yh$ is Nisnevich-$\A^1$-local, we have from Lemma \ref{lemmaA1local} 
a bijection
\[
\Hom_{\Hhset(S)}(X, \Yh) \xrightarrow{\cong} \Hom_{\Hh_{\et}(S)}(X, \Yh); 
\]
this proves the assertion.
\end{proof}

\subsection{Motivic $E$-cohomology groups} We use Proposition \ref{verdierA1prop} 
to prove the Verdier Hypercovering Theorem \ref{verdierE} for cohomology theories arising 
from motivic spectra. Let $\Pro^1$ be the projective line over $S$ pointed at $\infty$. 
Recall that a motivic or $\Pro^1$-spectrum over $S$ is a sequence 
$
E=(E_0, E_1, \ldots )
$ 
of pointed spaces  $E_n\in \Spcp(S)$, together with bonding maps
$
\sigma_n \colon E_n \wedge \Pro^1 \to E_{n+1}
$
in $\Spcp(S)$. 
A morphism $f \colon E \to F$ of $\Pro^1$-spectra is a sequence of maps 
$f_n \colon E_n \to F_n$ in $\Spcp(S)$ which commute with the bonding maps. We write $\Spt(S)$ 
for the category of motivic spectra. 

Given an object $\Xh\in \Spcp(S)$, one can associate to $\Xh$ its motivic suspension 
spectrum, which is given by the sequence of pointed spaces
\[
\Sigma^{\infty}_{\Pro^1}(\Xh):= (\Xh, \Xh \wedge \Pro^1, \ldots, \Xh \wedge (\Pro^1)^{\wedge n}, \ldots )
\]
together with the identity maps as bonding maps. This suspension yields 
a functor $\Sigma^{\infty}_{\Pro^1} \colon \Spcp(S) \to \Spt(S)$, which has 
a right adjoint $\Omega_{\Pro^1}^{\infty} \colon \Spt(S) \to \Spcp(S)$. 
Starting with a model structure on spaces, one obtains via a formal process a stable 
model structure on $\Spt(S)$ such that suspension with $\Pro^1$ induces an equivalences 
of categories. If we equip $\Spcp(S)$ with the Nisnevich (resp.~\'etale) motivic model 
structure, we obtain the stable Nisnevich (resp.~\'etale) model structure on $\Spt(S)$. 
Let $\SHh_{Nis}(S)$ (resp.~$\SHh_{\et}(S)$) be the corresponding Nisnevich (resp.~\'etale)
stable motivic homotopy category. Then the above pair of adjoint functors becomes 
a Quillen pair of adjoint functors 
\[
\Sigma_{\Pro^1}^{\infty} \colon \Spcp(S) \rightleftarrows \Spt(S) \colon \Omega_{\Pro^1}^{\infty}. 
\]
 
Recall that there are other suspension operators which play a role in the definition of 
generalized motivic cohomology groups. For example, if $K$ is a simplicial set, considered 
as a constant presheaf, then $K$ defines a space in $\Spc(S)$ (also denoted by $K$). For 
$K=S^1$ the simplicial circle, one defines a simplicial suspension operator 
$\Sigma_s \colon \Spcp(S) \to \Spcp(S)$ by the formula
\[
\Xh \mapsto S^1 \wedge \Xh.
\]
Also, for $\Ge_m = \A^1-\{0\}$ over $S$ pointed at $1$, one sets $\Sigma_{\Ge_m}\Xh= \Ge_m \wedge \Xh$,
and given integers $m\ge n$ one defines the motivic sphere $S^{m,n}\in \Spc(S)$ by
\[
S^{m,n}= \Sigma_s^{m-n}\Sigma_{\Ge_m}^n(S^0). 
\]
These three suspension operators are related in $\Hhp(S)$ by the isomorphisms
\[
\Pro^1 \cong S^1 \wedge \Ge_m = S^{2,1},
\]
which show that the suspensions $\Sigma_s$ and $\Sigma_{\Ge_m}$ become invertible in the 
stable motivic homotopy category. Thus it makes sense to define $S^{m,n}$ for all integers 
$m,n$; we write $E \mapsto \Sigma^{m,n}E$ for the induced operator on spectra. 

If $\Xh$ is a space, let $\Xh_+$ be the pointed space obtained by attaching a disjoint 
base point. Given a $\Pro^1$-spectrum $E$, the motivic (or Nisnevich motivic) 
$E$-cohomology groups of the space $\Xh$ with respect to $E$ are defined as the groups
\[
E^{m,n}(\Xh) = \Hom_{\SHh_{\Nis}(S)}(\Sigma_{\Pro^1}^{\infty}(\Xh_+), \Sigma^{m,n}E),
\]
Analogously, the \'etale motivic $E$-cohomology groups of $\Xh$ are given by
\[
E_{\et}^{m,n}(\Xh) = \Hom_{\SHh_{\et}(S)}(\Sigma_{\Pro^1}^{\infty}(\Xh_+), \Sigma^{m,n}E).
\]

We now prove Theorem \ref{verdierE}. 

\begin{proof}(of Theorem \ref{verdierE})
Let $E$ be a $\Pro^1$-spectrum which is fibrant in the Nisnevich stable motivic model structure 
on $\Spt(S)$. Being fibrant in the Nisnevich (resp.~\'etale) stable model 
structure means that $\Eh$ consists of Nisnevich (resp.~\'etale) $\A^1$-local 
spaces $E_n$ such that if ${\mathcal Hom}$ is the internal function object, 
the bonding maps induce weak equivalences 
\[
E_n \to {\mathcal Hom}(\Pro^1, E_{n+1}). 
\]
By Lemma \ref{lemmaA1local} finding a fibrant replacement of $E$ in the 
\'etale stable motivic model structure on $\Spt(S)$ only requires to take functorial fibrant 
replacements of the spaces $E_n$ in the \'etale local model structure on $\Spcp(S)$. 
Let $\Eh(n)[m]$ be the Nisnevich $\A^1$-local space $\Omega_{\Pro^1}^{\infty}(\Sigma^{m,n}E)$. Then 
$\Eh(n)[m]$ represents $E$-cohomology in $\Hh_{\Nis}(S)$,~i.e.~for every $\Xh$ we have for the 
$E$-cohomology groups
\[
E^{m,n}(\Xh) = \Hom_{\Hh_{\Nis}(S)}(\Xh, \Eh(n)[m]). 
\]
By the previous remark on fibrant spectra, we know that taking a functorial fibrant replacement 
of the spaces $E_n$ in the \'etale local model structure on $\Spcp(S)$ yields a fibrant replacement 
in in the stable \'etale motivic model structure of $\Eh$, and hence also a fibrant replacement 
of $\Eh(n)[m]$. This implies that $\Eh(n)[m]$ also represents the \'etale $E$-cohomology groups 
in $\Hh_{\et}(S)$,~i.e.
\[
E_{\et}^{m,n}(\Xh) = \Hom_{\Hh_{\et}(S)}(\Xh, \Eh(n)[m]).
\]
The Theorem follows now from Proposition \ref{verdierA1prop}, applied with $\Yh=\Eh(n)[m]$. 
\end{proof}

\begin{remark}
We remark that Theorem \ref{verdierE} does not state that motivic $E$-cohomology satisfies \'etale 
descent in the sense of Thomason \cite{thomason}. One can formulate such an \'etale descent 
statement for a motivic spectrum $E$ as follows: Let $\alpha$ be the change of topology morphism 
from the \'etale to the Nisnevich site. There is a pair of adjoint functors
\[
\alpha^* \colon \SHh_{\Nis}(S) \rightleftarrows \SHh_{\et}(S) \colon R_{\et}\alpha_*,
\]
and a motivic spectrum $E \in \SHh_{\Nis}(S)$ satisfies \'etale descent, if the adjunction
\bq\label{honestdescent}
\Eh \to R_{\et}\alpha_*\alpha^*\Eh
\eq
is an equivalence. Note that \eqref{honestdescent} is not an equivalence in general; for example, 
algebraic $K$-theory with finite coefficients satisfies \'etale descent only after 
inverting a Bott element, see \cite{thomason} and 
\cite{bott}. 
\end{remark}

\section{\'Etale algebraic and Hodge filtered cobordism}

In this section, we let $S=\Spec(\C)$ be the spectrum of the field $\C$ of complex 
numbers. We use Theorem \ref{verdierE} to construct maps from \'etale algebraic cobordism
(represented by Voevodsky's motivic Thom spectrum $MGL$ \cite{voevodsky}) to Hodge filtered 
cobordism (represented by the spectrum $\MUlog$ \cite{hfc}). Since the construction of 
$\MUlog$ is rather technical, we will only briefly introduce the properties needed for 
the proof below; for details we refer the reader to \cite{hfc}. 

Let $S^1$ be the simplicial circle, viewed as a constant presheaf, and let $\Spt_s(\C)$ be 
the category of $S^1$-spectra in $\Smc$. Thus objects of $\Spt_s(\C)$ are sequences  
$F=(F_0, F_1, \ldots )$ of pointed spaces $F_n$, together with bonding maps
$F_n \wedge S^1 \to F_{n+1}$ in $\Spc_*(\C)$. 
We consider $\Spc(\C)$ with the Nisnevich local model structure and denote by 
$\SHh_{s, \Nis}(\Smc)$ the homotopy category of the induced stable model structure. 

Given an integer $n$, we have in the category $\Spt_s(\Smc)$ morphisms
%\bq\label{MUlogdef}
%\xymatrix{
%& Rf_*MU \ar[d] \\
%H(A^{n+*}_{\log}(\pi_{2*}MU\otimes \C)) \ar[r] & Rf_*H(A^*(\pi_{2*}MU\otimes \C)).}
%\end{equation}
\begin{equation}\label{MUlogdef}
H(A^{n+*}_{\log}(\pi_{2*}MU\otimes \C)) \rightarrow Rf_*H(A^*(\pi_{2*}MU\otimes \C))\leftarrow Rf_*MU
\end{equation}
and the $S^1$-spectrum $\MUlog(n)$ is defined as the homotopy pullback resulting from these data. 
By construction, suitable suspensions of the objects $Rf_*MU$, $Rf_*H(A^*(\pi_{2*}MU\otimes \C))$
and  $H(A^{n+*}_{\log}(\pi_{2*}MU\otimes \C))$ represent in $\SHh_{s, \Nis}(\C)$ 
complex cobordism, singular cohomology and certain levels of the 
Hodge filtration respectively. The wedge of the spectra $\MUlog(n)$ for all integers $n$ 
defines a spectrum $\MUlog$ in $\Spt_s(\C)$ which represents (logarithmic) Hodge filtered cobordism 
in $\SHh_{s, \Nis}(\C)$. By \cite[Theorem 7.6 and Proposition 7.9]{hfc}, Hodge filtered cobordism 
is an oriented motivic cohomology theory on $\Smc$ and is represented by a $\Pro^1$-spectrum 
in the stable Nisnevich motivic homotopy category, which we also denote by $\MUlog$. The 
motivic Hodge filtered cobordism groups of a space $\Xh \in \Spc(\C)$ are the groups represented 
by this spectrum
\[
\MUlog^m(n)(\Xh) = \Hom_{\SHh_{\Nis}(\C)}(\Sigma_{\Pro^1}^{\infty}(\Xh_+), \Sigma^{m,n}\MUlog). 
\]
%{\color{red} INDICES!!!!!?????}

We prove Theorem \ref{thm-maps}.

\begin{proof}(of Theorem \ref{thm-maps}) 
Since $\MUlog$ is an oriented motivic cohomology theory \cite[Proposition 7.9]{hfc}, 
it follows from the universal property of algebraic cobordism \cite[Theorem 1.1]{ppr} 
that there is a canonical map in the motivic stable category
\[
MGL \to \MUlog. 
\]
In particular, given an \'etale hypercover $\Ub\rightarrow X$, we have natural maps 
\bq\label{mgltomulogmap}
MGL^{m,n}(\Ub)  \to \MUlog^{m}(n)(\Ub).
\eq
Taking the colimit over all such hypercovers, Theorem \ref{verdierE} yields the map
\bq\label{mlgtomulogmapcolim}
\MGLet^{m,n}(X) \cong \colim_{\Ub \to X}MGL^{m,n}(\Ub)  \to \colim_{\Ub \to X}\MUlog^{m}(n)(\Ub). 
\eq
Thus we get maps as in (\ref{thm-maps-1claim}), provided Hodge filtered cobordism satisfies 
\'etale descent,~i.e.~for every \'etale hypercover $\Ub\rightarrow X$ we have an isomorphism
$$
\MUlog^{m}(n)(\Ub) \cong \MUlog^m(n)(X).$$
In order to show this, it suffices to show that each of the objects appearing in \eqref{MUlogdef} 
satisfies \'etale descent. Note that the topological realization functor sends an \'etale 
hypercover $\Ub \to X$ to a topological hypercover $f^{-1}(\Ub) \to f^{-1}(X)$. By 
\cite[Proposition 4.10 and Theorem 5.2]{di}, this map induces a weak equivalence of 
topological spaces 
\[
\hocolim f^{-1}(\Ub) \overset{\sim}{\to} f^{-1}(X),
\] 
which shows that the two objects representing complex cobordism and complex cohomology 
satisfy \'etale descent. It remains to check the Hodge filtered part of cohomology. 
%By \cite[Theorem 10.2]{dhi}, we can 
%assume that $\Ub$ is a simplicial object on $\Smc$. [OBVIOUS, SINCE \'ETALE HYPERCOVER?].  
%This implies that $\Ub(\C) \to X(\C)$ 
%is a simplicial object in the category of complex manifolds. 
For each component $U_n$ of $\Ub$ let $U_n\rightarrow X_n$ be a smooth compactification  
such that $Y_n=X_n\setminus U_n$ is a normal crossing divisor. The resulting simplicial 
scheme $\Xb$ is a smooth proper hypercover of $X$, and as described in \cite[(8.1.19), (8.1.20), 
and (8.3.3)]{hodge3}, the Hodge filtration on the cohomology of the simplicial scheme $\Ub$ 
induces the Hodge filtration on the cohomology of $X$. Moreover, the spectral sequence 
which relates the cohomology of the components $U_n$ with the cohomology of $\Ub$ is compatible with 
the Hodge filtration. Since cohomology with complex coefficients satisfies \'etale descent, 
this spectral sequence abuts to the complex cohomology of $X$ and degenerates at the $E_2$-term. 
Hence $\hocolim f^{-1}(\Ub) \to f^{-1}(X)$ also induces an isomorphism on Hodge filtered 
cohomology groups, which completes the construction of the maps in (\ref{thm-maps-1claim}). 
It is clear that these maps extend the maps from (\ref{algcob-fil}). 

The diagram \eqref{natdiagram} is induced by the map of motivic spectra $MGL \to H\Z$ and 
the fact that the complex realization $f^{-1}$ of this map in the topological stable 
homotopy category is equal to $MU \to H\Z$. Moreover, it has been shown in \cite{hfc} 
that the map $MU \to H\Z$ induces the indicated map from Hodge filtered cobordism to 
Deligne cohomology. 
The commutativity of diagram \eqref{natdiagram} follows from the universality of $MGL$ 
in the stable motivic homotopy category and the fact that the horizontal maps in 
\eqref{natdiagram} are defined via the colimit of $MGL^{m,n}(\Ub)$ for all \'etale 
hypercovers $\Ub \to X$, together with the isomorphism (\ref{db-deligne}). 
\end{proof}

%
%%%%%%%%%%%%%%%%%%%%%%%%%%%%%%%%%%%%%%%%%
%
%%%%%%%%%%%%%%%%%%%%%%%%%%%%%%%%%%%%%%%%%
%
\bibliographystyle{amsalpha}

\end{document}